\documentclass[12pt,oneside]{amsart}
\usepackage{amssymb,amsmath,enumitem,array}
\usepackage{multirow,booktabs,geometry,graphicx,mathtools,tikz-cd}
\usepackage{xcolor}
\allowdisplaybreaks

\colorlet{mdtRed}{red!50!black}
\definecolor{dblue}{rgb}{0,0,0.6}
\colorlet{dullmagenta}{magenta!50!black}
 \usepackage{hyperref}

\usepackage{hyperref}
\hypersetup{
    colorlinks,
    citecolor=blue,
    filecolor=dullmagenta,
    linkcolor=dblue,
    urlcolor=mdtRed
}

\newtheorem{theorem}{Theorem}[section]
\newtheorem{lemma}[theorem]{Lemma}
\newtheorem{proposition}[theorem]{Proposition}
\newtheorem{corollary}[theorem]{Corollary}

\newtheorem{question}[theorem]{Question}

\theoremstyle{definition}

\newtheorem{remark}[theorem]{Remark}
\newtheorem{example}[theorem]{Example}

\numberwithin{equation}{section}

\DeclareMathOperator{\rk}{Rk}
\DeclareMathOperator{\HF}{HF}

\everymath{\displaystyle}

\title{Waring rank of certain complete intersection forms}

\author[Bhat]{Meghana Bhat}
\address[Meghana Bhat]{College of Computing and Data Science, Nanyang Technological University, 50 Nanyang Ave, Block N 4, Singapore - 639798}
\email{meghana.bhat@ntu.edu.sg}

\author[Carlini]{Enrico Carlini}
\address[Enrico Carlini]{DISMA Department of Mathematical Sciences, Politecnico Di Torino,  Corso Duca degli Abruzzi, 24, Torino- 10129, Italy}
\email{enrico.carlini@polito.it}

\author[Masuti]{Shreedevi K. Masuti}
\address[Shreedevi K. Masuti]{Department of Mathematics, Indian Institute of Technology Dharwad, Permanent Campus, Chikkamalligwad,  Dharwad - 580011, Karnataka, India}
\email{shreedevi@iitdh.ac.in}

\author[Sebastian]{Ronnie Sebastian}
\address[Ronnie Sebastian]{Department of Mathematics, Indian Institute of Technology Bombay, Powai, Mumbai - 400076, Maharashtra, India}
\email{ronnie@iitb.ac.in}

\thanks{MB is supported by the SUG Grant WBS 03INS002590C130 funded by NTU Singapore, and at the beginning of this project MB was supported by the Prime Minister's Research Fellowship (PMRF), Govt. of India. EC acknowledge that this study was carried out within the PRIN 2022 PNRR project grant P2022E2Z4AK “0-Dimensional Schemes, Tensor Theory, and Applications” funded by European Union– Next Generation EU within the PRIN 2022 program (D.D. 104- 02/02/2022 Ministero dell’Università e della Ricerca). 
The project was partially supported by SPARC grant SPARC/2019-2020/P1566/SL}

\date{\today}
\begin{document}
\maketitle
 \tableofcontents
\begin{abstract}
The complete intersection forms are the homogeneous polynomials whose annihilator ideal or the Macaulay dual, is an Artinian complete intersection ideal. In this paper, we compute the Waring rank of certain complete intersection forms. As an application, we compute the Waring rank of binomials whose annihilator is a complete intersection ideal. Further, we also compute the Waring rank of the product of binary forms in different sets of variables.
\end{abstract}

\section{Introduction}
     Let $S=\mathbb{C}[x_0, \ldots, x_n] = \bigoplus_{i \geq 0}S_i$ be the standard graded polynomial ring, and  $F \in S$ be a homogeneous polynomial of degree $d,$ also known as a degree $d$ form. It is well known that there exist $L_1, \ldots, L_r \in S_1$ such that \[F=\sum_{i=1}^{r}\alpha_iL_i^d,\]  for some $\alpha_i \in \mathbb{C},$ $1 \leq i \leq r$  (see \cite{IK99}).
    Such a decomposition of $F$ is called  {\it Waring decomposition} or  {\it sum of power decomposition} of $F$. The least integer $r$ such that there exists a Waring decomposition of $F$ with exactly $r$ summand is  called the {\it Waring rank} of $F$ denoted by $\rk F.$

    Even though computing the Waring rank of a given form is an interesting problem in itself,
    it also has significant applications in diverse areas of mathematics like Algebraic Statistics, Signal Processing, Computational Complexity, Quantum Information Theory, etc \cite{BC12, QuantumInformationTheory, Lan12, Lan14GCT,Lan22,  LT10}. 
    \vskip 2mm

There are interesting articles in the literature on computing the Waring rank of a family of forms. For instance,
 a celebrated result by J. Alexander and A. Hirschowitz determines the rank of a generic form \cite{AH95}, and the Waring rank of monomials is computed in \cite{BBT13, CCG12}. 
 The Waring rank of a few more families is computed in \cite{CCCGW18, GG24}.
     In \cite{BM21}, the Waring rank of binomials in two variables is computed. Recently, the Waring decomposition of special binomials is studied in \cite{CGM25} motivated by their connections with complexity theory. This article is
     inspired by the recent works \cite{AltafiEtall25,GR25,Shibata25} on the characterization of a binomial $F$ such that the corresponding apolar algebra 
 $T/F^\perp$ where $T=\mathbb{C}[X_0,\ldots,X_n]$ is a complete intersection (see Section \ref{Sec:Background} for the definition of $F^\perp$).
  In fact, they showed that the annihilator ideal $F^{\perp}$ of such a binomial $F$ is a {\it quasi-monomial complete intersection} in ``most" of the cases, that is, an ideal of the form
  \[
  F^{\perp}=(X_0^{a_0}+G, \, X_1^{a_1}, \, X_2^{a_2}, \, \ldots, \, X_n^{a_n})
  \]
  where $G$  is a homogeneous form in $T$ of degree $a_0$ with $\deg_{X_0} G < a_0$, where $\deg_{X_0}G$ denotes the degree of $X_0$ in $G$.   Motivated by these results, 
  we call a non-zero form a {\emph{complete intersection form}} if its annihilator ideal is a complete intersection, and
 we ask the following question:

    \begin{question}\label{Question:WaringRkCompleteIntersection}
        Compute the Waring rank of complete intersection forms.
    \end{question}
    Any binary form (form in two variables) is a complete intersection form. In \cite{sylvester1851lx, Sylvester1851}, Sylvester gives an algorithm to compute the rank of binary forms. In \cite{TW15}, the Waring rank of Vandermonde forms \[V_n = \prod_{1 \leq i<j\leq n}(x_i -x_j)\] is computed and $\rk(V_n) = (n-1)!$. The Vandermonde forms are also complete intersection forms (see discussion before \cite[Proposition 4.17]{CCCGW18}). Note that the computation of other notions of rank, such as cactus and smoothable rank (see \cite{RS11, IK99} for definition), of complete intersection forms follows easily from a proof similar to \cite[Corollary 2]{RS11}. Surprisingly, the Waring rank of complete intersection forms remains a difficult problem.
    
    In this paper, we answer Question \ref{Question:WaringRkCompleteIntersection} partially (Theorem~\ref{thm2} and Theorem \ref{thm1}). As a consequence, we obtain the Waring rank of forms whose annihilator is a quasi-monomial complete intersection (Theorem \ref{Thm:MainRank}). 
    {Further, we use these results to compute the Waring rank of binomials whose annihilator ideal is a complete intersection (Theorem \ref{Thm:BinomialRank})}, and compute the Waring rank of the product of binary forms in different sets of variables (Theorem \ref{Thm:ProductBinary}).

    Our techniques of proof are motivated by the computation of the Waring rank of monomials. For complete intersection forms, we first employ e-compatibility to establish a definitive lower bound on the rank. We then prove that this bound is sharp by constructing a set of distinct points within the annihilator ideal whose cardinality matches this bound in certain cases.  Hence, we achieve the rank in this case. Moreover, to apply the result on the Waring rank of quasi-monomials in the case of the Waring rank of binomials, we need to know whether a certain generator of the annihilator ideal is square-free, which in general is not an easy task. Thanks to the results of \cite{GR25,Shibata25}, which provide a nice description of the generators and that indeed helped our computations.

\vskip 2mm

     The paper is organized as follows: In Section \ref{Sec:Background}, we recall some useful results from the apolarity theory. 
    In Section \ref{Sec:ComputationofWRCI}, we first compute some colon ideals which are useful in giving a lower bound on the Waring rank of $F $ using e-compatibility.
   Then we prove the main result of the paper, where we compute the Waring rank of complete intersection forms satisfying some additional conditions. As an application, in Section \ref{sec:Rank of binomials} we determine the Waring rank of binomials which are complete intersection forms. In addition, we also compute the Waring rank of forms that are the products of a binary form and a monomial in different sets of variables, and the product of binary forms in two different sets of variables in Section \ref{Sec:product-binary-forms}.
\vskip 2mm

In this paper, we have used Macaulay2 \cite{M2}, Singular \cite{Singular}, and the library \cite{EliasCode} for various computations.
\vskip 2mm
\noindent {\bf Acknowledgements}: MB and SKM thank the Politecnico di Torino for the hospitality. MB is grateful to EC for giving her an opportunity to visit him in April-May 2025, when this project began. MB also thanks IIT Bombay for hospitality in July 2026, where MB and RS further discussed and improved the results. EC thanks IIT Dharwad for the kind hospitality during the phases of this research.

\section{Background}\label{Sec:Background}
In this section, we recall a few well-known results in the literature that we use in the subsequent sections. 
We first recall results related to the Waring rank of forms. An important tool in the study of the Waring rank is the apolarity action. 
    Let $S=\mathbb{C}[x_0, \ldots, x_n]$ be the standard graded polynomial ring over $\mathbb{C}.$ The apolarity action of $T$ on  $S$ defined on the monomials as
    \[X_0^{a_0} \cdots X_n^{a_n} \circ (x_0^{b_0}\cdots x_n^{b_n}) = \frac{\partial^{a_0}}{\partial x_0^{a_0}}\cdots \frac{\partial^{a_n}}{\partial x_n^{a_n}}(x_0^{b_0}\cdots x_n^{b_n}),\]
    and then extended linearly on $S.$ 
    Given a homogeneous form $F \in S_d,$  we define the {\it annihilator ideal}
    \[F^{\perp} = \{ g \in T \mid g \circ F = 0\}.\] 
    Given a linear form $L = \sum_{i=0}^{n}a_{i}x_i \in S$,  
    we associate a point $P=[a_{0}: \cdots:a_{n}] \in \mathbb{P}^n,$ and vice-versa. 
    The following lemma describes the usefulness of this correspondence in the Waring decomposition.

\begin{lemma}\cite[Lemma 1.15]{IK99}\label{Lemma:Apolarity}  \rm(Apolarity Lemma)
      Let $\mathbb{X}=\{P_1, \ldots, P_r\} \subset \mathbb{P}^n$ be a set of points. Then  $F = \sum_{i=1}^r\alpha_i L_i^d$ for some $\alpha_i \in \mathbb{C}$ if and only if $I(\mathbb{X}) \subset F^{\perp}$ where $L_i$ is the linear form corresponding to the point $P_i$ for $1 \leq i \leq r.$
    \end{lemma}

    Any $\mathbb{X}$ such that $I(\mathbb{X}) \subset F^{\perp}$ is called an {\it apolar set } of $F.$
    Note that the Apolarity Lemma is useful in obtaining an upper bound for the Waring rank of $F.$ But obtaining a lower bound is also a difficult task. In \cite{CCCGW18},  Carlini, Catalisano, Chiantini, Geramita and Woo introduced the idea of e-compatibility, which is very useful to obtain lower bounds for the Waring rank of forms.

\begin{theorem}\cite[Theorem 3.3]{CCCGW18}\label{Thm:Ecomputability} 
    Let $F \in S_d,$ and $\mathbb{X}\in \mathbb{P}^n$ be an apolar set of $F.$ Let $I$ be an ideal in $T$ generated in degree $e >0,$  and $t \in I_e$. If $t$ is a non-zerodivisor in $T/(I(\mathbb{X})\colon I)$, then for $s >>0,$ we have
    \[ |\mathbb{X}| \geq \frac{1}{e}\sum_{i=0}^s\HF(T/I(\mathbb{X})\colon I +(t), i) \geq \frac{1}{e} \sum_{i=0}^s\HF(T/F^{\perp}\colon I +(t), i).\]
\end{theorem}

\vskip 2mm

\section{Waring rank of certain complete intersection forms}\label{Sec:ComputationofWRCI}
Let $T=\mathbb{C}[X_0, \ldots, X_n]$ be the 
standard graded polynomial ring. We first prove a colon identity which is useful in the computation of a lower bound of the Waring rank of complete intersection forms.
    
\begin{lemma}\label{Lemma:ColonRegularSequence}
    Let $h, f_2, \ldots, f_n$ be a regular sequence 
    of homogeneous polynomials in $T$, and 
    $J=(f_2, \ldots, f_n)$. For any homogeneous polynomial $a\in T$ we have
    \begin{align*}
        (J+ (ha)) :(h)= (J:(h)) +(a).
    \end{align*}
\end{lemma}
\begin{proof}
    Let $g\in (J+ (ha)) :(h)$. Therefore
    $gh \in J+ (ha)$. That is, 
    $gh = b' + a'ha$ for some $b' \in J$ and $a' \in T.$
    Since $h$ is a nonzerodivisor on $T/J$, we obtain
    $g-a'a \in J$. 
    Hence $g = b+a'a$ for some $b \in J$. Therefore $g \in (J:(h)) +(a)$. Verifying the other way containment is an easy check.
\end{proof}

We now prove a lower bound for the Waring rank of complete intersection forms $F$ with at least one of the generators of $F^{\perp}$ being non-reduced.
    
\begin{proposition}\label{Theorem:Rank CI non reduced}
    Let $F^\perp = (f_0, f_1, \ldots,f_n) \subseteq T$ be a complete intersection.
    Let $\deg f_i = a_i$ and assume that  $a_0 \leqslant a_1\leqslant \ldots \leqslant a_n$. 
    Assume that $f_i$ is non-reduced for some $i$. Then
    \[
    \rk(F) \geqslant \frac{1}{a_k}\prod\nolimits_{i=0}^n a_i\,,
    \] 
    where $k$ is the smallest such that $f_k$ is not reduced.
\end{proposition}
\begin{proof}
    \noindent   
    Let $f_k =h^2 g$ and $I=(h)$ with $\deg h \geq 1$. Suppose $\mathbb{X}\subset \mathbb P^n$
    is an  apolar set of $F$ such that $I(\mathbb{X})\subset F^\perp$.
    
    First we show that $h$ 
    is not a zero divisor on $T/(I(\mathbb{X}) : I)$. 
    If $h\in F^\perp=(f_0,\ldots,f_n)$, then since ${\rm deg}(h)<{\rm deg}(f_k)$,
    it follows that $f_k = h^2g$ is in the ideal $(f_0,\ldots,f_{k-1})$,
    a contradiction. Thus, $h \notin F^\perp$, hence $h \notin I(\mathbb{X})$.
    Suppose there exists $g_1 \in T$ such that 
    \[\bar{h}\bar{g_1} = \overline{0} \]
    in $T / (I(\mathbb{X}) : I)$,
    then $h^2g_1\in I(\mathbb X)$. Since $I(\mathbb X)$
    is reduced, we get $hg_1\in I(\mathbb X)$, which shows that $g_1 \in (I(\mathbb X):I)$.
    Hence $h$ is not a zero divisor on $T/(I(\mathbb{X}) : I)$.

    Now we apply \cite[Theorem 3.3]{CCCGW18} by taking $I=(h)$ and $t=h$.
    Let 
    $$\tilde{I} = (F^{\perp}:(h)) +(h)\,.$$
    Then we get 
    $$|\mathbb X|\geqslant \frac{1}{{\rm deg}(h)}\sum_{i\geqslant 0} \HF(T/\tilde I, i)\,.$$
    
    Let $J=(f_0,f_1,\ldots, f_{k-1},f_{k+1}, \ldots, f_n)$. 
    Then $F^{\perp} = J+(h^2g)$. Since $f_0,\ldots, f_{k-1}, h, f_{k+1}, \ldots, f_n$ 
    is also a regular sequence, applying Lemma \ref{Lemma:ColonRegularSequence} we obtain
    \begin{align*}
        \tilde{I}&= (F^{\perp}:(h)) +(h)\\
            &=((J+(h^2g)):(h))+(h)\\
            &=((J:(h))+(hg))+(h)\\
            &=(J:(h))+(h)\\
            &=(f_0,\ldots, f_{k-1}, h, f_{k+1}, \ldots, f_n)\,
    \end{align*} 
    as $h$ is regular on $T/J$.
    Hence 
    \begin{equation}\label{eq:e-comp}
    |\mathbb X|\geqslant \frac{1}{{\rm deg}(h)}\sum_{i\geqslant 0} \HF(T/\tilde I, i) = \prod_{i=0, i\neq k}^n \deg f_i  =  \frac{1}{a_k}\prod_{i=0}^n a_i \,. 
    \end{equation}
    This shows that 
    $$\rk(F) \geqslant  \frac{1}{a_k}\prod_{i=0}^n a_i\,.$$
    This completes the proof of the Proposition.
\end{proof}

Using the same idea as in the proof of \cite[Lemma 4.13]{CCCGW18}, we prove the following 
Proposition. For a homogeneous ideal $I$, we shall denote by $I_a$ the space 
of forms in $I$ of degree $a$. 

\begin{proposition}\label{upper bound k-1 reduced}
    Let $F^\perp = (f_0, f_1, \ldots,f_n) \subseteq T$ be a complete intersection.
    Let $\deg f_i = a_i$ and assume that  $a_0 \leqslant a_1\leqslant \ldots \leqslant a_n$. 
    For some  $0 \leq k \leq n$, if $V(f_0,f_1,\ldots,f_{k-1})\subset \mathbb P^n$ is a reduced subscheme, then 
    $$\rk(F) \leqslant  \frac{1}{a_k}\prod\nolimits_{i=0}^n a_i\,.$$
\end{proposition}
\begin{proof}
    Let $Y=V(f_0,f_1,\ldots,f_{k-1})$ and  
    let $Z\subset Y$ denote the singular locus. 
    If $k=0$ then we mean $Y=\mathbb P^n$ and so 
    $Z=\emptyset$. 
    
    By our hypothesis, we get that 
    ${\rm dim}(Y)=n-k$ and ${\rm dim}(Z)\leqslant n-k-1$. 
    Consider the space of forms $(F^{\perp})_{a_n}$. This defines a base point free 
    linear system on $Y$ and so also on $Z$. 
    Let $U:=Y\setminus Z$. For general $g_n$ in $(F^{\perp})_{a_n}$ we see that 
    $U\cap V(g_n)$ is smooth and ${\rm dim}(Z\cap V(g_n))\leqslant n-k-2$.
    Next consider the space of forms ${(F^{\perp})}_{a_{n-1}}$. This defines a base point 
    free linear system on $Y\cap V(g_n)$. For general $g_{n-1}$ in ${(F^{\perp})}_{a_{n-1}}$ we see that 
    $U\cap V(g_n,g_{n-1})$ is smooth and ${\rm dim}(Z\cap V(g_n,g_{n-1}))\leqslant n-k-3$.
    Proceeding in this fashion we see that 
    $U\cap V(g_n,\ldots,g_{k+1})$ is smooth and 
    ${\rm dim}(Z\cap V(g_n,\ldots,g_{k+1}))\leqslant -1$, that is, it is empty.
    It follows that 
    $$\mathbb X:=V(f_0,f_1,\ldots,f_{k-1},g_{k+1},\ldots,g_n)\subset \mathbb P^n$$
    is a smooth subscheme of dimension 0. Clearly, $I(\mathbb X)\subset F^\perp$ and 
    $|\mathbb X|=\frac{1}{a_k}\prod\nolimits_{i=0}^n a_i$. Since $\rk(F)\leqslant |\mathbb X|$, 
    this completes the proof of the Proposition.  
\end{proof}

We now prove the main result of this section.
\begin{theorem}\label{thm2}
    Let $F^\perp = (f_0, f_1, \ldots,f_n) \subseteq T$ be a complete intersection.
    Let $\deg f_i = a_i$ and assume that  $a_0 \leqslant a_1\leqslant \ldots \leqslant a_n$. 
    Assume that $f_i$ is non-reduced for some $i$. Let $k$ be the smallest 
    such that $f_k$ is not reduced and assume that the subscheme 
    $V(f_0,f_1,\ldots,f_{k-1})\subset \mathbb P^n$
    is reduced for $0 \leq k \leq n$. Then $\rk(F)=\frac{1}{a_k}\prod\nolimits_{i=0}^n a_i$.
\end{theorem}
\begin{proof}
    This follows from Proposition \ref{Theorem:Rank CI non reduced}
    and Proposition \ref{upper bound k-1 reduced}.
\end{proof}
\begin{remark}
    The hypothesis that $V(f_0,f_1,\ldots,f_{k-1})\subset \mathbb P^n$ is a reduced subscheme
    is necessary. See Example \ref{counterexample}.
\end{remark}

The following result shows that if $V(f_0, \ldots, f_{n-1})$ is reduced, then without the assumption of non-reducedness on $f_n$, we can compute the Waring rank of $F$.

\begin{theorem}\label{thm1}
    Let $F^\perp=(f_0,f_1,\ldots,f_n)$ be a complete intersection. 
    Let $a_i={\rm deg}(f_i)$ and assume that $a_0\leqslant a_1\leqslant \ldots\leqslant a_n$.
    If the subscheme $V(f_0,f_1,\ldots,f_{n-1})\subset \mathbb P^n$
    is reduced, then 
    $\rk(F)=\prod\nolimits_{i=0}^{n-1}a_i$.
\end{theorem}
\begin{proof}
    From Proposition \ref{upper bound k-1 reduced} it follows that 
    $\rk(F)\leqslant \prod\nolimits_{i=0}^{n-1}a_i$. 
    Since $F^{\perp}$ is generated in degree $a_n$, 
    by \cite[Proposition 1]{RS11} we have any scheme apolar 
    to $F$ has length at least  $ \prod\nolimits_{i=0}^{n-1}a_i.$ 
    Hence $\rk(F)\geqslant \prod\nolimits_{i=0}^{n-1}a_i$.  
    This completes the proof of the Theorem. 
\end{proof}
The following result is a consequence of Theorem \ref{thm2}.
\begin{corollary}\label{Thm:MainRank}
    Let $F \in T$ be such that $F^{\perp}=(X_0^{a_0}+G, \, X_1^{a_1}, \, X_2^{a_2}, \, \ldots, \, X_n^{a_n})$ where $G$ is a form of degree $a_0$ with $ \deg_{X_0} G <a_0$ and $2 \leq a_1 \leq a_2 \leq \cdots \leq a_n.$ Let $\lambda = \min\{a_0, a_1\}$.
    Then
    \[\rk F =   \begin{dcases}
                    {\frac{1}{\lambda}{\displaystyle\prod_{j=0}^n a_j}} & \text{ if } (X_0^{a_0}+G) \mbox{ has a multiple factor. } \\
                    & \\
                     {a_0{\displaystyle\prod_{j=2}^n a_j}} & \mbox{if } (X_0^{a_0}+G) \mbox{ does not have a multiple factor. }
                \end{dcases}\]
\end{corollary}
\begin{proof}
    We divide the proof into the following cases: 
    \newline\noindent{\underline{\it Case}}(i). Assume $(X_0^{a_0}+G)$ has a multiple factor. In this case, all the generators of $F^{\perp}$ are non-reduced. By applying Theorem \ref{thm1} with $k=0$ case, we obtain the result. 
    \newline\noindent{\underline{\it Case}}(ii). Assume $(X_0^{a_0}+G)$ does not have a multiple factor. In this case, we consider the following sub-cases:
    \newline{\underline{\it Case}}(ii.a). Suppose $a_0 \leq a_1$. Since $(X_0^{a_0}+G)$  is square-free, $V(X_0^{a_0}+G)$ is reduced. Moreover, $X_1^{a_1}$ is non-reduced. Hence, we apply Theorem \ref{thm2} with $k=1$ case, we obtain
    \[
    \rk F = {a_0{\displaystyle\prod_{j=2}^n a_j}}.
    \]
    {\underline{\it Case}}(ii.b). Suppose $a_0 > a_1$. In this case, the least degree generator of $F^{\perp}$ is $X_1^{a_1}$ which is non-reduced. Hence, we apply Theorem \ref{thm2} with $k=0$ case, we obtain \[
    \rk F = {a_0{\displaystyle\prod_{j=2}^n a_j}}.
    \]
\end{proof}

The following example demonstrates Corollary \ref{Thm:MainRank}.
\begin{example}
    Consider $I=(X^2+Y^2+Z^2, Y^3, Z^3)$ and $F= x^5-5x^3y^2-5x^3z^2+15xy^2z^2$.  By Macaulay2 we verify that $F^{\perp} = I.$ Hence, by Corollary \ref{Thm:MainRank} we obtain $\rk F = 6.$
\end{example}

\section{Waring rank of binomials whose annihilator ideal is a CI}\label{sec:Rank of binomials}
In this section,  we compute the Waring rank of binomials whose annihilator ideal is a complete intersection. In addition, we compute the Waring rank of forms that are products of a binary form and a monomial in a different set of variables. Recently in \cite{AltafiEtall25} the authors have computed the annihilator ideal of binomial forms in three variables. Subsequently,  the papers \cite{GR25, Shibata25} characterize the binomials whose annihilator ideal is a complete intersection. In fact, they have computed $F^\perp$ in this case. 
This computation shows that such an ideal is a quasi-monomial complete intersection in ``most of the cases''. As a consequence of Theorem \ref{Thm:MainRank} we compute the Waring rank of binomials in these cases.  Further, we compute the Waring rank of all binomial forms whose annihilator ideal is a complete intersection (CI).
Consider a binomial 
\[F = x_0^{a_0}x_1^{a_1}\cdots x_n^{a_n}(x_0^{b_0} \cdots x_l^{b_l} - x_{l+1}^{b_{l+1}} \cdots x_n^{b_n}). \]  A necessary condition (see \cite{GR25, Shibata25}) for $F^{\perp}$ to be a CI is $l=0.$ Therefore we consider the binomial of type $$F =x_0^{a_0}x_1^{a_1}\cdots x_n^{a_n}(x_0^{b_0} - x_1^{b_1}\cdots x_n^{b_n}).$$ Rearranging the variables if required, we can write $F = F_1F_2$ where $F_1 = x_0^{a_0}x_1^{a_1}\cdots x_k^{a_k}(x_0^{b_0} - x_1^{b_1}\cdots x_k^{b_k})$ is a binomial with  $b_i \neq 0$ for $1 \leq i \leq k $, $b_{k+1}= \cdots = b_n = 0$, and $F_2 = x_{k+1}^{a_{k+1}}\cdots x_n^{a_n}$ for some $1 \leq k \leq n$.  Moreover, by \cite[Proposition 3.1]{DFP20-ApolarAlgebra} we have
\begin{align}\label{EQ:SumofPerp}
    F^{\perp} = (F_1^{\perp})_{S_1} + (F_2^{\perp})_{S_2}
\end{align}
where $S_1 = \mathbb{C}[x_0, \ldots, x_k]$ and $S_2 = \mathbb{C}[x_{k+1}, \ldots, x_n].$
Note that $(F_2^{\perp})_{S_2}$ is a monomial complete intersection. Hence  $F^{\perp}$ is a CI whenever $(F_1^{\perp})_{S_1}$ is a CI.
If $k=1,$ then $F_1$ is a binary binomial and $(F_1^{\perp})_{S_1}$ is always a complete intersection, but may not be a quasi-monomial complete intersection.
More generally, we compute the Waring rank when $F_1$ is any binary form and $F_2$ is a monomial.
Using a description of annihilator ideal given by K. Shibata in \cite{Shibata25}, we obtain the Waring rank of binary binomials, and hence reprove the result \cite[Theorem 3.1]{BM21}.
If $k \geq 2,$ then by \cite{GR25, Shibata25} we have a characterization for $(F_1)^{\perp}$ to be a CI. In fact, it is a quasi-monomial complete intersection. Hence, we use Corollary \ref{Thm:MainRank} to obtain the Waring rank.

We first compute the Waring rank of binary binomials.  For this, we recall a result by K. Shibata \cite{Shibata25} which computes the annihilator ideal of a binary binomial.
\begin{theorem}[\cite{Shibata25}, Lemma 2.11]\label{Thm:Shibata}
    Let  $S=\mathbb{C}[x_0, x_1]$, and $T=\mathbb{C}[X_0, X_1]$ be the standard graded polynomial rings. Let 
    \[
    F=x_0^{a_0}x_1^{a_1}(x_0^{b_0}-x_1^{b_0})
    \]
    where $b_0, a_0,a_1 \in \mathbb{N}$ with $b_0\geq 1,$ $0 \leq  a_0\leq a_1.$
    Let \[
    v = \min \{i \mid a_0+1 \leq i b_0\}, \quad w = \min\{i \mid a_1+1 \leq i b_0\}.
    \]
    Then 
    \[
    F^{\perp} = \begin{dcases}
        \left(X_0^{a_0+b_0+1}, \, M=\sum_{i=0}^vX_0^{ib_0}X_1^{a_1+1-ib_0}\right) \quad \quad \quad \quad \quad \quad\quad \quad \quad \quad \quad\textrm{  if } v<w; \\
        \left( P= \sum_{i=0}^v X_0^{ib_0}X_1^{(v-i)b_0}, \, Q= \sum_{i=0}^{v-1} X_0^{a_0+1-(v-1-i)b_0}X_1^{a_1+1-ib_0} \right) \quad \quad\textrm{if } v=w.
    \end{dcases}
    \]
\end{theorem}
It is easy to see that when the polynomials $M, P, Q$ are square-free. Indeed, 
\[M(X_0^{b_0} -X_1^{b_0}) = X_0^{(v+1)b_0}X_1^{a_1+1-vb_0} - X_1^{a_1+1+b_0}. \]
Whenever $a_1+1-vb_0 \geq 2$, we obtain $M$ has a multiple factor. Similarly,
\[
P(X_0^{b_0} -X_1^{b_0})= X_0^{(v+1)b_0} - X_1^{(v+1)b_0}.
\]
Hence $P$ is always square-free. Note that
\[
Q(X_0^{b_0} -X_1^{b_0}) = X_0^{a_0+1+b_0}X_1^{a_1+1-(v-1)b_0} - X_0^{a_0+1-(v-1)b_0}X_1^{a_1+1+b_0}.
\]
Hence $Q$ has a multiple factor if $a_0+1-(v-1)b_0 \geq 2$ or $a_1+1-(v-1)b_0 \geq 2,$ otherwise it is square free. By Sylvester's algorithm \cite{sylvester1851lx,Sylvester1851}, the result below is a direct consequence. 
\renewcommand{\arraystretch}{1.5}
\begin{theorem}\label{Thm:Binary}
    Let $F=x_0^{a_0}x_1^{a_1}(x_0^{b_0}-x_1^{b_0})$ be a binary binomial with the notation as in Theorem \ref{Thm:Shibata}. Then we have the following:
    \begin{center}
        \begin{tabular}{|c|c|c|}
        \hline
           $v<w$  & $a_1+1 < a_0+b_0+1$ and $a_1+1 -v b_0 \geq 2$ & $\rk F = a_0+b_0+1 $\\
           \cline{2-3}
             & Otherwise & $\rk F = a_1+1.$ \\
             \hline
             $v=w$ & $vb_0 \geq   a_0+a_1+2-(v-1)b_0$, $a_1+1-(v-1)b_0 \leq 1$ & $\rk F = a_0+a_1+2-(v-1)b_0$ \\
             \cline{2-3}
             & \textrm{Otherwise} & $\rk F = vb_0.$ \\
             \hline
        \end{tabular}
    \end{center}
\end{theorem}
 Note that the above result is equivalent to \cite[Theorem 3.1]{BM21}. Rewriting the statement in Theorem \ref{Thm:Binary} we get 
 $F = x_0^{a_0+b_0}x_1^{a_1} - x_0^{a_0}x_1^{a_1+b_0}$.
In terms of notation as in \cite[Theorem 3.1]{BM21}, we get
$r = a_0, \alpha = b_0, s = a_1.$ Under this identification, it is easy to verify that both results yield the same Waring rank.

The following lemma shows that, for a binary form, we can always assume the largest degree generator of the annihilator ideal is reduced.
\begin{lemma}\label{lem:binary}
    If $F \in \mathbb{C}[x_0,x_1]_{d}$, 
    then ${F}^{\perp} = (g_0, g_1)$ where $\deg (g_0) \leqslant \deg (g_1)$ and 
    $g_1=\ell^2 g$ for some linear form 
    $\ell$, $g\in \mathbb{C}[x_0,x_1]$.  
\end{lemma}
\begin{proof}
    It is a well-known fact that for $F \in \mathbb{C}[x_0,x_1]_{d}$, 
    we have ${F}^{\perp} = (g_0, g_1)$ such that $\deg (g_0) \leqslant \deg (g_1)$. We want to show that the generator $g_1$ can be replaced by a polynomial which is not reduced. 
    Since $g_0, g_1$ is a regular sequence, choose a linear form $\ell_0\mid g_0$ and $\ell_0 \nmid g_1.$ Let $\tilde {g_0} =\ell_0^{\deg g_1 - \deg g_0} g_0$. For any $\lambda \in \mathbb{C}\setminus\{0\}$, we have  $h = \tilde {g_0} - \lambda g_1 \in F^{\perp}$ and $F^{\perp} = (g_0, h)$.
    
    We now find $\lambda \in  \mathbb{C}\setminus\{0\}$ such that $h$ is not reduced.
    Consider the map $\varphi: \mathbb P^1\to \mathbb P^1$
defined by $\varphi([a:b])= [\tilde {g_0}(a,b):g_1(a,b)]$. Note that $\varphi$ is a proper map since the graphs of $\tilde {g_0}, g_1$ are closed. Moreover, for each point $P=[u:v] \in \mathbb{P}^1$, the fibre at $P$, $\varphi^{-1}(P)=V_+(v\tilde{g_0} -ug_1)$ is finite where $V_+(v\tilde{g_0} -ug_1)$ denotes the projective zero set of the homogeneous polynomial $v\tilde{g_0} -ug_1$. Hence $\varphi$ is a finite map of degree equals to $\deg g_1$.  By Hurwitz formula, see \cite[Chapter 4, Corollary 2.4]{Ha},
the number of ramification points of this map is nonzero, and we can choose a ramification point $P$ such that $g_1(P) \neq 0.$ 
We now claim that $h$ is not reduced for $\lambda = \frac{\tilde {g_0}(P)}{g_1(P)}$. Note that \[\varphi^{-1}([\lambda:1]) = V_+(\tilde{g_0} - \lambda g_1) = V_+(h).\]
Since $P=[u:v]$ is ramified, we get $\mid \varphi^{-1}([\frac{\tilde {g_0}(P)}{g_1(P)}:1])\mid < \deg \varphi.$ Hence $\mid V_+(h) \mid < \deg g_1$. Therefore $h$ has a multiple factor $\ell = vx+uy$. i.e., we can assume $g_1= h =\ell^2g$ for some $g \in \mathbb{C}[x,y]$. 
\end{proof}

    We now present the computation of the Waring rank of $F = F_1F_2$, where $F_1$ is a binary form and $F_2$ a monomial in Equation \eqref{EQ:SumofPerp}.  Indeed, this includes the case where $F_1$ is a binary binomial. 

\begin{theorem}
    Let $F =F_1F_2$  where $F_1 \in \mathbb{C}[x_0, x_1]_d$, and $F_2=x_2^{a_2}\cdots x_n^{a_n}$ such that $n \geq 2$, and $a_2 \leq \cdots \leq a_n$ be integers. Then
    \begin{align}\label{EQ:rkFormula}
        \rk F = (\rk F_1)\cdot \max\{a_2+1, \deg F_1+2-\rk F_1\} \cdot\left(\prod_{i=3}^n (a_i+1)\right).
    \end{align}
\end{theorem}
\begin{proof}
    By Lemma \ref{lem:binary} assume $F_1^{\perp} = (g_0, g_1)$ with $g_1 = \ell^2g$ with $\deg(g_0) \leq \deg(g_1)$. Then
    \[
    F^{\perp} = (g_0, \ell^2g, X_2^{a_2+1}, \ldots, X_n^{a_n+1}).
    \]
  We now divide the proof in to following cases and apply Theorem \ref{thm2} to obtain the result. 
  \newline \noindent {\underline {\it Case} (i)}. Suppose $g_0$ is reduced and $\deg g_0 \leq a_2+1.$ Hence $V(g_0)$ is reduced. Note that except $g_0$, all the other generators are non-reduced in $F^{\perp}.$
  Hence by Theorem \ref{thm2} we obtain 
  \[\rk F = \left(\frac{1}{\min\{a_2+1, \deg g_1\}}\right)\deg g_0 \cdot \deg g_1 \cdot \prod_{i=2}^n(a_i+1). \]
  Since $g_0$ is reduced, $\deg g_0 = \rk F_1,$ and $\deg g_1 = \deg F_1+2\rk F_1$. Therefore we obtain
  \begin{align*}
        \rk F = (\rk F_1)\cdot \max\{a_2+1, \deg F_1+2-\rk F_1\} \cdot\left(\prod_{i=3}^n (a_i+1)\right).
    \end{align*}

\noindent {\underline {\it Case} (ii)}. Suppose $g_0$ is reduced and $\deg g_0 > a_2+1.$ In this case the least degree generator of $F^{\perp}$ is $X_2^{a_2+1}$ which is non-reduced. We apply Theorem \ref{thm2} with $k=0$ case to obtain 
\[
\rk F = \frac{1}{(a_2+1)} \deg g_0 \cdot \deg g_1 \cdot \prod_{i=2}^n(a_i+1).
\]
As in the above case (i), since $g_0$ is reduced, $\rk F_1 = \deg g_0$ and $\deg g_1 = \deg F_1 + 2 \rk F_1$ and $\deg g_1 \geq \deg g_0 > a_2+1.$ Hence we obtain the Equation \eqref{EQ:rkFormula}.

\noindent{\underline {\it Case} (iii)} Assume $g_0$ is non-reduced. In this case, all the generators of $F^{\perp}$ are non-reduced. 
By applying Theorem \ref{thm2} with $k=0$ case, we obtain
\[
\rk F = \left(\frac{1}{\min\{a_2+1, \deg g_0\}}\right) \deg g_1 \cdot \prod_{i=2}^n(a_i+1).
\]
Moreover, $\rk F_1= \deg g_1$ and $\deg g_0 = \deg F_1 +2 -\rk F_1.$ Substituting in the Equation above, we obtain the result.
\end{proof}

We now present the case $k\geq 2$ in Equation \eqref{EQ:SumofPerp} with $F_1$ is a binomial such that $F_1^{\perp}$ is a CI and $F_2$ is a monomial. By the results of \cite[Theorem 4.1]{AltafiEtall25},  \cite[Theorem 3.3]{GR25}, and \cite[Theorem 1.3]{Shibata25} we know that if $F^{\perp}$ is a complete intersection, then it is a quasi-monomial complete intersection. Hence $\rk F$  is {\it almost} immediate from Theorem \ref{Thm:MainRank}.
First, we state the known results describing $F^{\perp}$.
\begin{theorem}[{\cite[Theorem 4.1]{AltafiEtall25}, \cite[Theorem 3.3]{GR25}, \cite[Theorem 1.3]{Shibata25}}]\label{Thm:BinomialAnnihilator}
 Let  
\[ F = x_0^{a_0}\cdots x_k^{a_k}(x_0^{b_0} - x_1^{b_1}\cdots x_k^{b_k})\]
be a binomial form with all $b_i \neq 0$ such that $\sum_{i=1}^n b_i = b_0$ and $k \geq 2,$ and $a_i \geq 0.$ If there exists an integer $1 \leq i \leq k$ with $a_i < q b_i$ where $q = \lfloor{\frac{a_0+1}{b_0}}\rfloor$, then
\[
F^{\perp} = (X_0^{a_0+1}+G, X_1^{a_1+b_1+1}, \ldots , X_k^{a_k+b_k+1}) 
\]
where $G = \sum_{j=1}^m (X_1^{b_1}\cdots X_k^{b_k})^{j} X_0^{a_0+1 - jb_0}$ with 
$m = \min \{\lfloor {\frac{a_i}{b_i}} \rfloor \mid 1 \leq i \leq k\}.$
    
\end{theorem}

We now present our results on the Waring rank of binomials whose annihilator ideal is a quasi-monomial complete intersection as in Theorem \ref{Thm:BinomialAnnihilator}.

\begin{theorem}\label{Thm:BinomialRank}
    Let $F =x_0^{a_0}x_1^{a_1}\cdots x_n^{a_n}(x_0^{b_0} - x_1^{b_1}\cdots x_n^{b_n})$.  Write $F = F_1F_2$ where $F_1 = x_0^{a_0}x_1^{a_1}\cdots x_k^{a_k}(x_0^{b_0} - x_1^{b_1}\cdots x_k^{b_k})$ a binomial with  $b_i \neq 0$ for $1 \leq i \leq k $, $b_{k+1}= \cdots = b_n = 0$, and $F_2 = x_{k+1}^{a_{k+1}}\cdots x_n^{a_n}$ for some $1 \leq k \leq n$, and  $a_{k+1} \leq \cdots \leq a_n$. 
    \newline If $2 \leq k \leq n$, and $F_1$ satisfies conditions as in Theorem \ref{Thm:BinomialAnnihilator}, then  
    \[
    \rk F = \frac{(a_0+1)}{\lambda} \prod_{i=1}^k (a_i+b_i+1)\prod_{i=k+1}^n (a_i+1) 
    \]
    where $\lambda = \min\{\{a_0+1, a_k+1\} \cup \{a_i+b_i+1 \mid 1 \leq i \leq k\}\}$. 
\end{theorem}    
    
\begin{proof}
   By Theorem \ref{Thm:BinomialAnnihilator} we have 
   \begin{align*}
       F^{\perp} = (X_0^{a_0+1}+G, X_1^{a_1+b_1+1}, \ldots , X_k^{a_k+b_k+1}, X_{k+1}^{a_{k+1}+1}, \ldots, X_n^{a_n+1}) 
   \end{align*}
  where $G = \sum_{j=1}^m (X_1^{b_1}\cdots X_k^{b_k})^{j} X_0^{a_0+1 - jb_0}$ with 
$m = \min \{\lfloor {\frac{a_i}{b_i}} \rfloor \mid 1 \leq i \leq k\}$. Hence $F^{\perp}$ is a quasi-monomial complete intersection.
We {\it claim} that $G$ is not square-free. That is, we prove that $a_0+1 -mb_0 \geq 2.$ 
    On the contrary, assume that $a_0+1 - mb_0 \leq 1.$ This means
    \begin{equation}\label{EQ:qLeqm}
        q = \lfloor \dfrac{a_0+1}{b_0}\rfloor \leq \lfloor \frac{mb_0+1}{b_0}\rfloor \leq \lfloor m +\frac{1}{b_0} \rfloor = m.
    \end{equation}
    The last equality follows from the fact that $b_0 = \sum_{i=1}^k b_i \geq k \geq 2.$ 
    Moreover, by the definition of $m$ we have $m b_j \leq a_j$ for all $1 \leq j \leq k.$
    By hypothesis there exists $i$ such that $a_i < q b_i$ for some $1 \leq i \leq k.$ 
Therefore, for the same $i$, we have $mb_i \leq a_i <qb_i.$ Since $b_i >0$ we get
    \begin{equation}\label{EQ:mLessthanq}
        m <q.
    \end{equation}
    Hence, by Equations \eqref{EQ:qLeqm} and \eqref{EQ:mLessthanq} we have a contradiction. Therefore, we can write $G = X_0^{a_0+1-mb_0} g$ for some $g\in T.$ Thus $X_0^{a_0+1}+G$ is not square-free. Hence, by Corollary \ref{Thm:MainRank}, we obtain the result.
    \end{proof}

 \noindent We apply the Corollary \ref{Thm:BinomialRank} in a specific case in the example below:
    \begin{example}
        Let $F= x_0^2x_1^3x_2^3x_3^4(x_0^6-x_1^4x_2x_3) \in \mathbb{C}[x_0, x_1, x_2, x_3]$. By Theorem \ref{Thm:BinomialRank}, we obtain the Waring rank $\rk F = 240$.
        
    \end{example}

\section{Waring rank of product of binary forms}\label{Sec:product-binary-forms}
In this section we compute the Waring rank of the product 
of binary forms in different set of variables.
Throughout this section, assume that 
$S=\mathbb{C}[x_0,x_1, y_0,y_1]$, 
$T=\mathbb{C}[X_0,X_1, Y_0,Y_1]$, 
$T_1= \mathbb{C}[X_0,X_1]$, and $T_2= \mathbb{C}[Y_0,Y_1]$.   
Let $F_1 \in (S_1)_{d_1}$, and $F_2 \in (S_2)_{d_2}$, where 
$d_1,d_2>0$. 
For $1 \leqslant i \leqslant 2$,  we know that $(F_i)^{\perp} \in T_i$ is a complete intersection. Let
\[
F_1^\perp = (g_0, g_1) \subset T_1 \quad \text{with } \deg (g_0)  \leqslant\deg (g_1),
\]
and similarly
\[
F_2^\perp = (h_0, h_1) \subset T_2, \quad \text{with } \deg (h_0)  \leqslant\deg (h_1).
\]
By Sylvester's algorithm \cite{sylvester1851lx,Sylvester1851}, we have
\begin{align}\label{eq:Sylvesters algo}
    \rk (F_1) &=   \begin{cases}
                \deg (g_0) & \text{if } g_0 \text{ is reduced} \\
                \deg (g_1)   & \text{otherwise.}
                 \end{cases}\\
            d_1 +2 &=\deg (g_0) + \deg (g_1)\nonumber
\end{align}
A similar result holds for $\rk (F_2)$.

Let $F = F_1 F_2 \in S_{d_1+d_2}$. By \cite[Proposition 3.1]{DFP20-ApolarAlgebra} we have
\[
F^{\perp} = (F_1^{\perp})_{T_1} + (F_2^{\perp})_{T_2}.
\]
Therefore
\[
F^\perp = (g_0, g_1, h_0, h_1) \subset T,
\]
is a complete intersection in $T$. Without loss of generality, assume that ${\rm deg}(g_0) \leqslant {\rm deg}(h_0)$.

As an application of results in Section \ref{Sec:ComputationofWRCI}, we prove the following result.
\begin{theorem}[Waring rank of product of binary forms]\label{Thm:ProductBinary}
    If $F=F_1F_2$ such that $F_1 \in (T_1)_{d_1}$, $F_2 \in (T_2)_{d_2}$ as defined above,
    then 
    \[
    \rk (F) = (\rk (F_1))\cdot (\rk (F_2)) \cdot\max\{d_1+2-\rk (F_1), d_2+2 -\rk (F_2)\}.
    \]
\end{theorem}
\begin{proof}
Without loss of generality, assume that ${\rm deg}(g_0) \leqslant {\rm deg}(h_0)$.
We prove the result by considering the following cases.
\begin{enumerate}[label={\rm (\Alph*)}]
    \item If $g_0$ is non-reduced. 
    Applying Theorem \ref{thm2} with $k=0$ we get
    \[
    \rk (F) ={\rm deg}(g_1){\rm deg}(h_0){\rm deg}(h_1).
    \]
    \item If $g_0$ is reduced, 
    ${\rm deg}(g_0)\leqslant {\rm deg}(g_1)\leqslant {\rm deg}(h_0)$. 
    By Lemma \ref{lem:binary}, $g_1$ is non-reduced, and 
    applying Theorem \ref{thm2} with $k=1$ we get 
    \[
    \rk (F) ={\rm deg}(g_0){\rm deg}(h_0){\rm deg}(h_1).
    \]
    \item If $g_0$ is reduced, 
    ${\rm deg}(g_0)\leqslant {\rm deg}(h_0)< {\rm deg}(g_1)$, 
    and $h_0$ is non-reduced.
    Applying Theorem \ref{thm2} with $k=1$ we get     
    \[
    \rk (F) ={\rm deg}(g_0){\rm deg}(g_1){\rm deg}(h_1).
    \]
    \item If $g_0$ is reduced, 
    ${\rm deg}(g_0)\leqslant {\rm deg}(h_0)< {\rm deg}(g_1)$, 
    and $h_0$ is reduced.
    By Lemma \ref{lem:binary}, both $g_1$ 
    and $h_1$ are non-reduced. 
    Applying Theorem \ref{thm2} with $k=2$ we get 
    \[
    \rk (F) ={\rm deg}(g_0){\rm deg}(h_0)\cdot {\rm max}\{{\rm deg}(g_1),{\rm deg}(h_1)\}.
    \]
\end{enumerate}

Using \eqref{eq:Sylvesters algo} we compute the entries
in the following table. Note that 
if $\rk (F_2) = \deg (h_i)$, then $d_2+2-\rk (F_2) = \deg h_{1-i}$.

\begin{center}
{\footnotesize
\begin{tabular}[t]{|p{.5cm}|p{5cm}|p{1.2cm}|p{1.2cm}|p{6cm}|}
		\hline
		   & \bf Hypothesis & $\mathbf{Rk}(F_1)$ & $\mathbf{Rk}(F_2)$ & $\mathbf{max}\{d_1+2-\mathrm{rk} F_1, d_2+2 -\mathrm{rk} F_2\}$ \\ \hline  
        (A) & \vspace{-1ex}
        \begin{itemize}[leftmargin=*, nosep]\setlength{\itemsep}{0pt}
        \item $g_0$ is non-reduced
        \end{itemize} & $\deg (g_1)$ & $\deg (h_i)$ & $\max\{\deg (g_0), \deg (h_{1-i})\} = \deg (h_{1-i})$ \\ \hline  
        (B) & \vspace{-1ex}
        \begin{itemize}[leftmargin=*, nosep]\setlength{\itemsep}{0pt}
                \item $g_0$ is reduced
                \item $\mathrm{deg}(g_0)\leqslant \mathrm{deg}(g_1)\leqslant \mathrm{deg}(h_0)$
              \end{itemize} & $\deg (g_0)$ & $\deg (h_i)$ &$\max\{\deg (g_1), \deg (h_{1-i})\} = \deg (h_{1-i})$  \\ \hline  
        (C) & \vspace{-1ex}
        \begin{itemize}[leftmargin=*, nosep]\setlength{\itemsep}{0pt}
                \item $g_0$ is reduced
                \item $\mathrm{deg}(g_0)\leqslant \mathrm{deg}(h_0)< \mathrm{deg}(g_1)$
                \item $h_0$ is non-reduced
              \end{itemize} &$\deg (g_0)$  & $\deg (h_1)$& $\max\{\deg (g_1), \deg (h_0)\} = \deg (g_1)$  \\ \hline  
        (D) & \vspace{-1ex}
        \begin{itemize}[leftmargin=*, nosep]\setlength{\itemsep}{0pt}
                \item $g_0$ is reduced
                \item $\mathrm{deg}(g_0)\leqslant \mathrm{deg}(h_0)< \mathrm{deg}(g_1)$
                \item $h_0$ is reduced
              \end{itemize}\vspace{-2ex} &$\deg (g_0)$  & $\deg (h_0)$& $\max\{\deg (g_1), \deg (h_1)\}$  \\ \hline  
\end{tabular}
}
\end{center}
One easily checks that in all cases 
$\rk (F) = (\rk (F_1))\cdot (\rk (F_2)) \cdot\max\{d_1+2-\rk (F_1), d_2+2 -\rk (F_2)\}$. This completes
the proof of the Theorem.
\end{proof}

\section{Final remarks}
The following example suggests that to get a closed formula for the Waring rank of complete intersection forms in terms of the degree of the generators of $F^{\perp}$ as in Theorem \ref{thm2}, the assumption of reducedness is perhaps necessary.
\begin{example}\label{counterexample}
Consider $I= (Y^2-Z^2, XY, X^5-Z^5) \subset T=\mathbb{C}[X,Y,Z]$. 
Note that $I$ is an Artinian complete intersection ideal. Hence, 
there exists a homogeneous form $F$ such that $F^{\perp}=I$. Moreover, each of the generators of $F^{\perp}$ is reduced(does not have a multiple factor), but $V(Y^2-Z^2, XY)$ is non-reduced.
We show that $4 < \rk (F) \leqslant 8.$

\begin{proof}
We first show that $\rk F \geq 4$. 
Let $\mathbb{X}$ be such that $I(\mathbb{X}) \subseteq F^{\perp}$.
Since $I(\mathbb X)$ is a radical ideal, by \cite{GMR83}, the sequence $(c_0,c_1,\ldots)$ where
\[c_i=\HF(T/I(\mathbb{X}), i)-\HF(T/I(\mathbb{X}), i-1)\]
is the Hilbert function of an Artinian ring. It follows
that 
$$\HF(T/I(\mathbb{X}), i)\geqslant \HF(T/I(\mathbb{X}), i-1)$$
for all $i$. 
Since $I(\mathbb{X}) \subset F^{\perp}$
we obtain 
$$\HF(T/F^{\perp}, i) \leq \HF(T/I(\mathbb{X}), i)$$ 
for all $i$.
Hence, 
$$3 = \HF(T/F^{\perp}, 1) \leq \HF(T/I(\mathbb{X}), 1)\leqslant 3,$$
and 
$$4 = \HF(T/F^{\perp}, 2) \leq \HF(T/I(\mathbb{X}), 2)\,.$$ 
Therefore, $\HF(T/I(\mathbb{X}), i) \geq 4$ for $i \geq 2.$
Recall that $\rk(F)=\HF(T/I(\mathbb{X}), i)$ for $i\gg 0$.
Thus, $\rk(F) \geq 4$.

We now claim that $\rk(F)> 4$.
If not, assume that $\rk(F)=|\mathbb X|=4$.
As $\HF(T/I(\mathbb{X}), i)$ is eventually constant and equal to 4,
it follows that the Hilbert function of $T/I(\mathbb{X})$ is $(1,3,4,4,\ldots)$.
Hence,  $I(\mathbb{X})_2 = (F^{\perp})_2 = (Y^2-Z^2, XY)$. 
Therefore, $I(\mathbb{X}) \supseteq (Y^2-Z^2, XY)$. Moreover, it is easy to check that the Hilbert function of $T/J$ is also $(1,3,4,4, \ldots )$, where $J=(Y^2-Z^2, XY)$. Hence $I(\mathbb{X}) =(Y^2-Z^2, XY)$ which is a contradiction because  $V(Y^2-Z^2, XY)$ is not reduced.

We give an upper bound by giving an
explicit apolar ideal $I(\mathbb{X}) \subset F^{\perp}$.
Let $I(\mathbb{X}) = (XY,Y(Y^2-Z^2), X^6+Z^4(Y^2-Z^2))$. 
Using Macaulay2 \cite{M2}, it is easy to check that 
$I(\mathbb{X}) \subset F^{\perp}$ and $I(\mathbb{X})$ 
is a radical ideal with $|\mathbb{X}| = 8$. 
Therefore, $\rk (F) \leqslant 8$.

\end{proof}
\end{example}

\bibliography{main}
\bibliographystyle{halpha}
\end{document}